\documentclass[reqno]{amsart}
\usepackage{amsmath, amsthm, amssymb, amstext}
\usepackage{hyperref,xcolor}
\definecolor{bluecite}{HTML}{0875b7}
\hypersetup{
	pdfborder={0 0 0},
	colorlinks,
	linkcolor=bluecite,
	citecolor=bluecite
}
\usepackage{enumitem}
\allowdisplaybreaks

\newtheorem{theorem}{Theorem}
\newtheorem{remark}[theorem]{Remark}
\newtheorem{lemma}[theorem]{Lemma}
\newtheorem{proposition}[theorem]{Proposition}

\newtheorem{definition}[theorem]{Definition}

\DeclareMathOperator*{\divergenz}{div}              %
\DeclareMathOperator*{\dist}{dist}          %

\newcommand{\N}{\mathbb{N}}

\newcommand{\R}{\mathbb{R}}

\newcommand*\diff{\mathop{}\!\mathrm{d}}

\newcommand{\Lp}[1]{L^{#1}(\Omega)}

\newcommand{\Wp}[1]{W^{1,#1}(\Omega)}

\newcommand{\Wpzero}[1]{W^{1,#1}_0(\Omega)}

\newcommand{\eps}{\varepsilon}

\newcommand{\into}{\int_{\Omega}}

\newcommand{\weak}{\rightharpoonup}

\newcommand{\Linf}{L^{\infty}(\Omega)}

\renewcommand{\l}{\left}
\renewcommand{\r}{\right}

\newcommand{\WH}{W^{1, \mathcal{H}}_0(\Omega)}
\newcommand{\LH}{L^{ \mathcal{H}}(\Omega)}
\numberwithin{theorem}{section}
\numberwithin{equation}{section}

\title[On a class of critical  double phase problems]{On a class of critical  double phase problems}

\author[C.\,Farkas]{Csaba Farkas}
\address[C.\,Farkas]{Sapientia Hungarian University of Transylvania, Department of Mathematics and Computer Science, T\^{a}rgu Mure\textcommabelow{s}, Romania}
\email{farkascs@ms.sapientia.ro, farkas.csaba2008@gmail.com}

\author[A.\,Fiscella]{Alessio Fiscella}
\address[A.\,Fiscella]{Dipartimento di Matematica e Applicazioni, Universit\`a degli Studi di Milano-Bicocca, Via Cozzi 55, Milano, CAP 20125, Italy}
\email{alessio.fiscella@unimib.it}

\author[P.\,Winkert]{Patrick Winkert}
\address[P.\,Winkert]{Technische Universit\"{a}t Berlin, Institut f\"{u}r Mathematik, Stra\ss e des 17.\,Juni 136, 10623 Berlin, Germany}
\email{winkert@math.tu-berlin.de}


\begin{document}

\begin{abstract}
	In this paper we study a class of double phase problems involving critical growth, namely
\begin{equation*}
	\begin{aligned}
		-\divergenz\big(|\nabla u|^{p-2} \nabla u+ \mu(x) |\nabla u|^{q-2} \nabla u\big)&=\lambda|u|^{\vartheta-2}u+|u|^{p^*-2}u && \text{in } \Omega, \\
		u&= 0&& \text{on } \partial\Omega,
	\end{aligned}
\end{equation*}
	where $\Omega \subset \R^N$ is a bounded Lipschitz domain, $1<\vartheta<p<q<N$, $q<p^*$ and $\mu(\cdot)$ is a nonnegative bounded weight function. The operator involved is the so-called double phase operator, which reduces to the $p$-Laplacian or the $(p,q)$-Laplacian when $\mu\equiv 0$ or $\inf \mu>0$, respectively. Based on variational and topological tools such as truncation arguments and genus theory, we show the existence of $\lambda^*>0$ such that the problem above has infinitely many weak solutions with negative energy values for any $\lambda\in (0,\lambda^*)$.
\end{abstract}

\maketitle

\section{Introduction}

In 1991, Garc\'{\i}a Azorero-Peral Alonso \cite{Garcia-Azorero-Peral-Alonso-1991} studied the critical $p$-Laplacian problem
\begin{equation}\label{problem2}
	\begin{aligned}
		-\divergenz\big(|\nabla u|^{p-2} \nabla u\big)&=|u|^{p^*-2}u+\lambda|u|^{q-2}u && \text{in } \Omega, \\
		u&= 0&& \text{on } \partial\Omega,
	\end{aligned}
\end{equation}
where $\lambda>0$, $1<p<N$, $1<q<p<p^*$ with $p^*$ being the critical Sobolev exponent given by
\begin{align}\label{critical-exponent}
	p^*=\frac{Np}{N-p}.
\end{align}
Using critical point theory and ideas of Benci-Fortunato \cite{Benci-Fortunato-1981} and of Garc\'{\i}a Azorero-Peral Alonso \cite{Garcia-Azorero-Peral-Alonso-1987}, problem \eqref{problem2} admits infinitely many solutions whenever $\lambda\in(0,\lambda_1)$ for $\lambda_1>0$ small enough, see Theorem 4.5 in \cite{Garcia-Azorero-Peral-Alonso-1991}. The main difficulty in their treatment is the lack of compactness in the embedding $\Wpzero{p} \hookrightarrow \Lp{p^*}$ and so the corresponding energy functional does not satisfy the Palais-Smale condition in general. Afterward, the techniques in \cite{Garcia-Azorero-Peral-Alonso-1991} have been transferred by several authors to different type of problems with critical growth. We refer to the works of Candito-Marano-Perera \cite{Candito-Marano-Perera-2015} for the $(p,q)$-Laplacian case with $(p,q)$-linear term, Corr\^{e}a-Costa \cite{Correa-Costa-2015} for Kirchhoff $p(x)$-Laplace problems,  Figueiredo-Santos J\'{u}nior-Su\'{a}rez \cite{Figueiredo-Santos-Junior-Suarez-2015} for anisotropic equations, Li-Zhang \cite{Li-Zhang-2009} and Yin-Yang \cite{Yin-Yang-2012} for $(p,q)$-Laplace equations with sublinear/superlinear nonlinearities and Zhang-Fiscella-Liang \cite{Zhang-Fiscella-Liang-2019} for fractional $p$-Laplacian Kirchhoff problems, see also the references therein.

Originally, the study of elliptic equations involving critical growth was initiated by the work of Br\'ezis-Nirenberg \cite{Brezis-Nirenberg-1983} who are concerned with the existence of positive solutions to the semilinear equation 
\begin{equation}\label{problem3}
	\begin{aligned}
		-\Delta u&=u^{p}+f(x,u)&& \text{in } \Omega, \\
		u&= 0&& \text{on } \partial\Omega,
	\end{aligned}
\end{equation}
where $p+1=2N/(N-2)$ is the critical Sobolev exponent of the embedding $H^1_0(\Omega)\hookrightarrow \Lp{p}$ and $f\colon\Omega\times\R\to\R$ is a lower-order perturbation of $u^p$ in the sense that $\displaystyle\lim_{s\to\infty} \frac{f(x,s)}{s^p}=0$. Problems of type \eqref{problem3} are motivated by variational problems in geometry and physics where the lack of compactness also occurs, for example, the Yamabe problem on Riemannian manifolds, see Yamabe \cite{Yamabe-1960} or the existence of nonminimal solutions for Yang-Mills functions, see Taubes \cite{Taubes-1982-1, Taubes-1982-2}. We refer to \cite{Brezis-Nirenberg-1983} for more examples.

In the current paper, we are interested of equations of type \eqref{problem2} which are driven by the so-called double phase operator which is given by
\begin{align}\label{double-phase-operator}
	\divergenz\big(|\nabla u|^{p-2} \nabla u+ \mu(x) |\nabla u|^{q-2} \nabla u\big)\quad \text{for }u\in \WH
\end{align}
with an appropriate Musielak-Orlicz Sobolev space $\WH$. It is easy to see that this operator is a generalization of the $p$-Laplacian and the $(p,q)$-Laplacian for $p<q$ by setting $\mu\equiv 0$ or $\inf \mu>0$, respectively. In 1986, Zhikov \cite{Zhikov-1986} studied the corresponding energy functional to \eqref{double-phase-operator} given by
\begin{align}\label{integral_minimizer}
	\omega \mapsto \int_\Omega \l(\frac{1}{p}|\nabla  \omega|^p+\frac{\mu(x)}{q}|\nabla  \omega|^q\r)\diff x
\end{align}
in order to provide models for strongly anisotropic materials, that is, the modulating coefficient $\mu(\cdot)$ dictates the geometry of composites made of two different materials with distinct power hardening exponents $p$ and $q$. Note that \eqref{integral_minimizer} is a prototype of a functional 
whose integrands change their ellipticity according to the points where $\mu(\cdot)$ vanishes or not. We refer to the papers of Baroni-Colombo-Mingione \cite{Baroni-Colombo-Mingione-2015,Baroni-Colombo-Mingione-2018}, Colombo-Mingione \cite{Colombo-Mingione-2015a,Colombo-Mingione-2015b} and  Marcellini \cite{Marcellini-1991,Marcellini-1989b} for deeper investigations of (local) minimizers of \eqref{integral_minimizer}, see also the paper of Mingione-R\u{a}dulescu \cite{Mingione-Radulescu-2021}  about recent developments for problems with nonstandard growth and nonuniform ellipticity.

Given a bounded domain $\Omega \subset \R^N$, $N\geq 2$ with Lipschitz boundary $\partial\Omega$, in this paper we consider the following quasilinear elliptic equation with critical growth
\begin{equation}\label{problem}
	\begin{aligned}
		-\divergenz\big(|\nabla u|^{p-2} \nabla u+ \mu(x) |\nabla u|^{q-2} \nabla u\big)&=\lambda|u|^{\vartheta-2}u+|u|^{p^*-2}u && \text{in } \Omega, \\
		u&= 0&& \text{on } \partial\Omega,
	\end{aligned}
\end{equation}
where $\lambda>0$ is a real parameter to be specified, $p^*$ is the critical exponent to $p$ given in \eqref{critical-exponent} and we suppose that
\begin{align}\label{assumptions}
	1<\vartheta<p<q<N, \quad q<p^*
	\quad\text{and}\quad
	0 \leq \mu(\cdot)\in \Linf.
\end{align}

The main result in this paper reads as follows.

\begin{theorem}\label{main_theorem}
	Let \eqref{assumptions} be satisfied. Then there exists $\lambda^*>0$ such that problem \eqref{problem} admits infinitely many weak solutions with negative energy values for any $\lambda \in (0,\lambda^*)$.
\end{theorem}

The proof of Theorem \ref{main_theorem} relies on a careful combination of variational and topological tools, such as truncation techniques and genus theory introduced by Krasnosel'skii \cite{Krasnoselskii-1964}. Indeed, in the study of problem \eqref{problem} there are lots of difficulties to be overcome. First, the corresponding energy functional to problem \eqref{problem} does not contain the norm of the space $\WH$, so the truncation argument introduced by Garc\'{\i}a Azorero-Peral Alonso \cite{Garcia-Azorero-Peral-Alonso-1991} in order to have a local Palais-Smale condition needs a more careful treatment. Second, in contrast to the works with the $(p,q)$-Laplacian, see \cite{Li-Zhang-2009} and \cite{Yin-Yang-2012}, because of \eqref{assumptions}, we are not working in usual Sobolev spaces but in Musielak-Orlicz Sobolev spaces. In this direction, even if we know that $\WH\hookrightarrow W_0^{1,p}(\Omega)$ continuously, the classical Lions' concentration-compactness principle in $W_0^{1,p}(\Omega)$ cannot work for \eqref{problem}. Indeed, because of the $q$-term appearing in the double phase operator, we need to know if for $u\in\WH$ we can reach $u\in L^{q^*}(\Omega)$, with $q^*=Nq/(N-q)$ being the critical Sobolev exponent at level $q$. However, the optimality of Sobolev type embeddings for $\WH$ is still an open problem. In order to handle the critical Sobolev nonlinearity in \eqref{problem}, we exploit a suitable convergence analysis of gradients, inspired by Boccardo-Murat \cite{Boccardo-Murat-1992}.

To the best of our knowledge there are only three works dealing with a double phase operator along with critical growth. In Farkas-Winkert \cite{Farkas-Winkert-2021} the authors study the singular double phase problem
\begin{equation}\label{problem4}
	\begin{aligned}
		-\divergenz (A(u)) & = u^{p^{*}-1}+\lambda \l(u^{\gamma-1}+ g(x,u)\r) \quad && \text{in } \Omega,\\
		u& = 0 &&\text{on } \partial \Omega,
	\end{aligned}
\end{equation}
with
\begin{align*}
	\divergenz (A(u)):=\divergenz\big(F^{p-1}(\nabla u)\nabla F(\nabla u) +\mu(x) F^{q-1}(\nabla u)\nabla F(\nabla u)\big)
\end{align*}
being the Finsler double phase operator and $(\R^N,F)$ stands for a Minkowski space; see also the corresponding nonhomogeneous Neumann case by the authors \cite{Farkas-Fiscella-Winkert-2021}. In \cite{Farkas-Winkert-2021} and \cite{Farkas-Fiscella-Winkert-2021}, the existence of at least one solution of \eqref{problem4} is shown by a local analysis on a suitable closed convex subset of $\WH$. In order to provide a multiplicity result for \eqref{problem}, in the current paper we need to work globally in the whole space $\WH$. Very recently, Crespo-Blanco-Papageorgiou-Winkert \cite{Crespo-Blanco-Papageorgiou-Winkert-2021} have been considered a nonhomogeneous singular Neumann double phase problem with critical growth on the boundary given by
\begin{equation*}
	\begin{aligned}
		-\divergenz\big(|\nabla u|^{p-2} \nabla u+ \mu(x) |\nabla u|^{q-2} \nabla u\big)+\alpha(x)u^{p-1}&= \zeta(x)u^{-\kappa}+\lambda u^{q_1-1} && \text{in } \Omega, \\
		\big(|\nabla u|^{p-2} \nabla u+ \mu(x) |\nabla u|^{q-2} \nabla u\big) \cdot \nu&= -\beta(x)u^{p_* -1} && \text{on } \partial\Omega.
	\end{aligned}
\end{equation*}
Based on the fibering method introduced by Dr\'{a}bek-Pohozaev \cite{Drabek-Pohozaev-1997} along with the Nehari manifold approach, the existence of at least two solutions is obtained.

Finally, we mention recent papers on the existence of solutions for double phase problems with homogeneous Dirichlet boundary condition treated by different methods in case of smooth or nonsmooth right-hand sides. We refer to Colasuonno-Squassina \cite{Colasuonno-Squassina-2016} for corresponding eigenvalue problems, Fiscella \cite{Fiscella-2020} involving Hardy potentials, Fiscella-Pinamonti \cite{Fiscella-Pinamonti-2020} for Kirchhoff type problems, Gasi\'nski-Papa\-georgiou \cite{Gasinski-Papageorgiou-2019} for locally Lipschitz continuous right-hand sides, Gasi\'nski-Winkert \cite{Gasinski-Winkert-2020a,Gasinski-Winkert-2020b,Gasinski-Winkert-2021} for convection and superlinear problems, Liu-Dai \cite{Liu-Dai-2018} for superlinear problems, Perera-Squassina \cite{Perera-Squassina-2019} for a Morse theoretical treatment, Zeng-Bai-Gasi\'nski-Winkert \cite{Zeng-Bai-Gasinski-Winkert-2020, Zeng-Gasinski-Winkert-Bai-2020} for multivalued obstacle problems and the references therein. 

\section{Preliminaries}

In this section we will recall the main properties of Musielak-Orlicz spaces $\Lp{\mathcal{H}}$, $\WH$ and the topological tools which are needed in our treatment. First, we denote by $\Lp{r}$ and $L^r(\Omega;\R^N)$ the usual Lebesgue spaces with the norm $\|\cdot\|_r$  and  by $\Wp{r}$ the corresponding  Sobolev spaces with norm $\|\cdot\|_{1,r}$ for $1\leq r \leq \infty$.

Let  $\mathcal{H}\colon \Omega \times [0,\infty)\to [0,\infty)$ be the nonlinear map defined by
\begin{align*}
	\mathcal H(x,t):= t^p+\mu(x)t^q,
\end{align*}
where we suppose \eqref{assumptions} and let $M(\Omega)$ be the space of all measurable functions $u\colon\Omega\to\R$. The Musielak-Orlicz Lebesgue space $L^\mathcal{H}(\Omega)$ is given by
\begin{align*}
	L^\mathcal{H}(\Omega)=\left \{u\in M(\Omega)\,:\,\varrho_{\mathcal{H}}(u)<\infty \right \}
\end{align*}
equipped with the Luxemburg norm
\begin{align*}
	\|u\|_{\mathcal{H}} = \inf \left \{ \tau >0\,:\, \varrho_{\mathcal{H}}\left(\frac{u}{\tau}\right) \leq 1  \right \},
\end{align*}
where the modular function is given by
\begin{align}\label{modular}
	\varrho_{\mathcal{H}}(u):=\into \mathcal{H}(x,|u|)\diff x=\into \big(|u|^{p}+\mu(x)|u|^q\big)\diff x.
\end{align}

The norm $\|\cdot\|_{\mathcal{H}}$ and the modular function $\varrho_{\mathcal{H}}$ are related as follows, see Liu-Dai \cite[Proposition 2.1]{Liu-Dai-2018} or Crespo-Blanco-Gasi\'nski-Harjulehto-Winkert \cite[Proposition  2.14]{Crespo-Blanco-Gasinski-Harjulehto-Winkert-2021}.

\begin{proposition}\label{proposition_modular_properties}
	Let \eqref{assumptions} be satisfied, $y\in \Lp{\mathcal{H}}$, $\zeta>0$ and $\varrho_{\mathcal{H}}$ be defined by \eqref{modular}. Then the following hold:
	\begin{enumerate}
		\item[\textnormal{(i)}]
		If $y\neq 0$, then $\|y\|_{\mathcal{H}}=\zeta$ if and only if $ \varrho_{\mathcal{H}}(\frac{y}{\zeta})=1$;
		\item[\textnormal{(ii)}]
		$\|y\|_{\mathcal{H}}<1$ (resp.\,$>1$, $=1$) if and only if $ \varrho_{\mathcal{H}}(y)<1$ (resp.\,$>1$, $=1$);
		\item[\textnormal{(iii)}]
		If $\|y\|_{\mathcal{H}}<1$, then $\|y\|_{\mathcal{H}}^q\leq \varrho_{\mathcal{H}}(y)\leq\|y\|_{\mathcal{H}}^p$;
		\item[\textnormal{(iv)}]
		If $\|y\|_{\mathcal{H}}>1$, then $\|y\|_{\mathcal{H}}^p\leq \varrho_{\mathcal{H}}(y)\leq\|y\|_{\mathcal{H}}^q$;
		\item[\textnormal{(v)}]
		$\|y\|_{\mathcal{H}}\to 0$ if and only if $ \varrho_{\mathcal{H}}(y)\to 0$;
		\item[\textnormal{(vi)}]
		$\|y\|_{\mathcal{H}}\to \infty$ if and only if $ \varrho_{\mathcal{H}}(y)\to \infty$.
	\end{enumerate}
\end{proposition}

Furthermore, we define the weighted space
\begin{align*}
	L^q_\mu(\Omega)=\left \{u\in M(\Omega)\,:\,\into \mu(x) |u|^q \diff x< \infty \right \}
\end{align*}
endowed with the seminorm
\begin{align*}
	\|u\|_{q,\mu} = \left(\into \mu(x) |u|^q \diff x \right)^{\frac{1}{q}}.
\end{align*}
While, the corresponding Musielak-Orlicz Sobolev space $W^{1,\mathcal{H}}(\Omega)$ is set as
\begin{align*}
	W^{1,\mathcal{H}}(\Omega)= \Big \{u \in L^\mathcal{H}(\Omega) \,:\, |\nabla u| \in L^{\mathcal{H}}(\Omega) \Big\}
\end{align*}
equipped with the norm
\begin{align*}
	\|u\|_{1,\mathcal{H}}= \|\nabla u \|_{\mathcal{H}}+\|u\|_{\mathcal{H}},
\end{align*}
where $\|\nabla u\|_\mathcal{H}=\|\,|\nabla u|\,\|_{\mathcal{H}}$. We denote by $W^{1,\mathcal{H}}_0(\Omega)$ the completion of $C^\infty_0(\Omega)$ in $W^{1,\mathcal{H}}(\Omega)$. By using \eqref{assumptions}, we know that we can endow the space $\WH$ with the equivalent norm given by
\begin{align*}
	\|u\|_{1,\mathcal{H},0}=\|\nabla u\|_{\mathcal{H}},
\end{align*}
see Proposition  2.18(ii) of Crespo-Blanco-Gasi\'nski-Harjulehto-Winkert \cite{Crespo-Blanco-Gasinski-Harjulehto-Winkert-2021}. Also, we have that the spaces $L^\mathcal{H}(\Omega)$, $W^{1,\mathcal{H}}(\Omega)$ and $\WH$ are uniformly convex and so reflexive Banach spaces, see  Colasuonno-Squassina \cite[Proposition 2.14]{Colasuonno-Squassina-2016} or Harjulehto-H\"{a}st\"{o} \cite[Theorem 6.1.4]{Harjulehto-Hasto-2019}.

Now, we recall the following embeddings for the spaces $\LH$ and  $\WH$, see Colasuonno-Squassina \cite[Proposition 2.15]{Colasuonno-Squassina-2016} or Crespo-Blanco-Gasi\'nski-Harjulehto-Winkert \cite[Proposition  2.16]{Crespo-Blanco-Gasinski-Harjulehto-Winkert-2021}.

\begin{proposition}\label{proposition_embeddings}
	Let \eqref{assumptions} be satisfied and let $p^*$ be the critical exponent to $p$. Then the following embeddings hold:
	\begin{enumerate}
		\item[\textnormal{(i)}]
		$\Lp{\mathcal{H}} \hookrightarrow \Lp{r}$ and $\WH\hookrightarrow \Wpzero{r}$ are continuous for any $r\in [1,p]$;
		\item[\textnormal{(ii)}]
		$\WH \hookrightarrow \Lp{r}$ is continuous for any $r \in [1,p^*]$ and compact for any $r \in [1,p^*)$;
		\item[\textnormal{(iii)}]
		$\Lp{\mathcal{H}} \hookrightarrow L^q_\mu(\Omega)$ is continuous;
		\item[\textnormal{(iv)}]
		$\Lp{q}\hookrightarrow\Lp{\mathcal{H}} $ is continuous.
	\end{enumerate}
\end{proposition}

\begin{remark}
	Note that Proposition \ref{proposition_embeddings}\textnormal{(ii)} holds for $r=q<p^*$ by \eqref{assumptions}. Thus, $\WH \hookrightarrow \Lp{q}$ is compact.
\end{remark}
\begin{remark}\label{remark-constants}
Throughout the paper, for any $r\in[1,p^*]$ we denote with $C_r>0$ the constant given by Proposition \ref{proposition_embeddings}\textnormal{(ii)}, such that
\begin{align*}
	\|u\|_r^r\leq C_r\|u\|_{1,\mathcal{H},0}^r
\end{align*}
for any $u\in \WH$.
\end{remark}

In order to prove Theorem \ref{main_theorem} we are going to use some topological results introduced by Krasnosel'skii \cite{Krasnoselskii-1964}. To this end, let $X$ be a Banach space and let $\Sigma$ be the class of all closed subsets $A\subset X\setminus\left\{0\right\}$ that are symmetric with respect to the origin, that is, $u\in A$ implies $-u\in A$.

\begin{definition}
	Let $A\in\Sigma$. The Krasnosel'skii's genus $\gamma(A)$ of $A$ is defined as being the least positive integer $n$ such that there is an odd mapping $\phi\in C(A,\mathbb R^n)$ such that $\phi(x)\neq0$ for any $x\in A$. If $n$ does not exist, we set $\gamma(A)=\infty$. Furthermore, we set $\gamma(\emptyset)=0$.
\end{definition}

The following proposition states the main properties on
Krasnosel'skii's genus which we need later, see Rabinowitz \cite{Rabinowitz-1986}.

\begin{proposition}\label{proposition-genus} 
	Let $A,B\in\Sigma$. Then the following hold:
	\begin{enumerate}
		\item[\textnormal{(i)}] 
			If there exists an odd continuous mapping from $A$ to $B$, then $\gamma(A) \leq \gamma(B)$;
		\item[\textnormal{(ii)}]
			If there is an odd homeomorphism from $A$ to $B$, then $\gamma(A) = \gamma(B)$;
		\item[\textnormal{(iii)}]
			If $\gamma(B) < \infty$, then $\gamma\overline{(A\setminus B)} \geq \gamma(A) - \gamma(B)$;
		\item[\textnormal{(iv)}]
			The $n$-dimensional sphere $S^n$ has a genus of $n+1$ by the Borsuk-Ulam Theorem;
		\item[\textnormal{(v)}]
			If $A$ is compact, then $\gamma (A) < \infty$ and there exists $ \delta> 0$ such that $N_{\delta} (A)\subset\Sigma $ and $\gamma (N_{\delta} (A)) =\gamma (A)$, with $N_\delta(A)=\left\{x\in X:\,\, \dist(x,A)\leq\delta\right\}$.
	\end{enumerate}
\end{proposition}

\section{Proof of the main result}

In this section we are going to prove Theorem \ref{main_theorem}. First, we note that the energy functional $J_\lambda\colon \WH\to\R$ related to problem \eqref{problem} is given by
\begin{align*}
	J_\lambda(u):=\frac{1}{p}\|\nabla u\|_p^p+\frac{1}{q}\|\nabla u\|_{q,\mu}^q-\frac{\lambda}{\vartheta}\|u\|_{\vartheta}^{\vartheta}-\frac{1}{p^{*}}\|u\|_{p^*}^{p^{*}}.
\end{align*}
It is clear that $J_\lambda \in C^1(\WH)$ and that the weak solutions of \eqref{problem} are exactly the critical points of $J_\lambda\colon \WH\to\R$.

Now, we discuss the compactness property for the functional $J_\lambda$, given by the Palais–Smale condition. We say that $\{u_n\}_{n\in\mathbb N}\subset\WH$ is a Palais-Smale sequence for $J_\lambda$ at level $c\in\mathbb R$ if
\begin{align}\label{palais-smale}
	J_\lambda(u_n)\to c \quad \text{and}\quad J'_\lambda(u_n)\to 0 \quad\text{in }\left(\WH\right)^*\quad\text{as }n\to\infty.
\end{align}
We say that $J_\lambda$ satisfies the Palais-Smale condition at level $c$ (\textnormal{(PS)$_c$} for short) if any Palais-Smale sequence $\{u_n\}_{n\in\N}$ at level $c$ admits a convergent subsequence in $\WH$.

\begin{lemma}\label{Lgradients}
	Let \eqref{assumptions} be satisfied and let $\{u_n\}_{n\in\N}\subset \WH$ be a bounded \textnormal{(PS)}$_c$ sequence with $c\in\R$. Then, up to a subsequence, $\nabla u_n(x)\to\nabla u(x)$ a.\,e.\,in $\Omega$ as $n\to\infty$.
\end{lemma}

\begin{proof}
	Since $\{u_n\}_{n\in\N}$ is bounded in $\WH$, by Proposition \ref{proposition_embeddings}(ii) and Br\'{e}zis \cite[Theorem 4.9]{Brezis-2011} along with the reflexivity of $\WH$, there exists a subsequence, still denoted by $\{u_n\}_{n\in\N}$, and $u\in \WH$ such that
	\begin{equation}\label{3.10}
		\begin{aligned}
			u_n&\weak u\quad \text{in }\WH, 
			\quad &\nabla u_n&\weak\nabla u \quad \text{in }\left[\LH\right]^N,\\
			u_n&\to u\quad \text{in }L^r(\Omega),
			\quad &u_n(x)&\to u(x)\quad \text{a.\,e.\,in }\Omega,\\
			|u_n(x)|& \leq h(x)\quad\text{a.\,e.\,in }\Omega,
		\end{aligned}
	\end{equation}
	as $n\to\infty$ with $r\in[1,p^*)$ and $h\in L^q(\Omega)$.
	
	For any $k\in\N$, let $T_k\colon\R\to\R$ be the truncation function defined by
	\begin{align*}
		T_k(t):=
		\begin{cases}
			t & \text{if } |t|\leq k,\\[1ex]
			\displaystyle k\frac{t}{|t|} &\text{if }|t|> k.
		\end{cases}
	\end{align*}
	Let $k\in\N$  be fixed. Then, since $\{u_n\}_{n\in\N}$ is a \textnormal{(PS)}$_c$ sequence for $J_\lambda$, we have
	\begin{equation}\label{LG1}
		\begin{aligned}
			o(1)&=\langle J'_\lambda(u_n),T_k(u_n-u)\rangle\\
			&= \int_\Omega\left(|\nabla u_n|^{p-2}\nabla u_n+\mu(x)|\nabla u_n|^{q-2}\nabla u_n\right)\cdot\nabla T_k(u_n-u)\,\mathrm{d}x\\
			&\quad-\lambda\int_\Omega|u_n|^{\vartheta-2}u_nT_k(u_n-u)\,\mathrm{d}x-\int_\Omega|u_n|^{p^*-2}u_nT_k(u_n-u)\,\mathrm{d}x,
		\end{aligned}
	\end{equation}
	as $n\to\infty$, because $\{T_k(u_n-u)\}_{n\in\N}$ is bounded in $\WH$.
	By H\"older's inequality, we see that the functional
	\begin{align*}
		G\colon g\in\left[L^{\mathcal H}(\Omega)\right]^N\mapsto\int_\Omega\left(|\nabla u|^{p-2}\nabla u+\mu(x)|\nabla u|^{q-2}\nabla u\right)\cdot g\,\,\mathrm{d}x
	\end{align*}
	is linear and bounded.
	From \eqref{3.10} we see that $\nabla T_k(u_n-u)\rightharpoonup0$ in $\left[L^{\mathcal H}(\Omega)\right]^N$, so we can get
	\begin{equation}\label{LG2}
		\lim_{n\to\infty}\int_\Omega\left(|\nabla u|^{p-2}\nabla u+\mu(x)|\nabla u|^{q-2}\nabla u\right)\cdot\nabla T_k(u_n-u)\,\mathrm{d}x=0.
	\end{equation}
	By the boundedness of $\{u_n\}_{n\in \N}$ and by Proposition \ref{proposition_embeddings}(ii), we also observe that for any $n\in\mathbb N$
	\begin{equation}\label{LG3}
		\left|\int_\Omega|u_n|^{p^*-2}u_nT_k(u_n-u)\,\mathrm{d}x\right|\leq k\int_\Omega|u_n|^{p^*-1}\,\mathrm{d}x\leq Ck
	\end{equation}
	with a constant $C>0$ independent of $n$ and $k$. Thus, by using \eqref{3.10}, \eqref{LG1} and \eqref{LG2}, we get
	\begin{equation}\label{LG4}
		\begin{aligned}
			&\limsup_{n\to\infty}\Bigg [ \int_\Omega\left[|\nabla u_n|^{p-2}\nabla u_n-|\nabla u|^{p-2}\nabla u\right] \cdot\nabla T_k(u_n-u)\,\mathrm{d}x \\
			& \qquad \qquad +\into 
			\mu(x)\left[|\nabla u_n|^{q-2}\nabla u_n-|\nabla u|^{q-2}\nabla u\right]\cdot\nabla T_k(u_n-u)\,\mathrm{d}x\Bigg ]\\
			&=\limsup_{n\to\infty}\int_\Omega|u_n|^{p^*-2}u_nT_k(u_n-u)\,\mathrm{d}x.
		\end{aligned}
	\end{equation}
	From Simon \cite[formula (2.2)]{Simon-1978} we have the well-known  inequalities

	\begin{equation}\label{LG5}
		(|\xi|^{r-2}\xi-|\eta|^{r-2}\eta)\cdot(\xi-\eta)\geq
		\begin{cases}
			\kappa_r|\xi-\eta|^r & \text{if }r\geq2,\\[1ex]
			\displaystyle \kappa_r\frac{|\xi-\eta|^2}{\left(|\xi|^r+|\eta|^r\right)^{\frac{2-r}{r}}} & \text{if }1<r<2,
		\end{cases}
	\end{equation}
	for any $\xi$, $\eta\in\R^N$ with a constant $\kappa_r>0$.
	Thus, by \eqref{LG3}, \eqref{LG4} and \eqref{LG5}, we obtain
	\begin{equation}\label{LG6}
		\begin{aligned}
			&\limsup_{n\to\infty}\int_\Omega\left(|\nabla u_n|^{p-2}\nabla u_n-|\nabla u|^{p-2}\nabla u\right)\cdot\nabla T_k(u_n-u)\,\mathrm{d}x\\
			&\leq
			\limsup_{n\to\infty}\int_\Omega|u_n|^{p^*-2}u_nT_k(u_n-u)\,\mathrm{d}x\\
			&\leq Ck.
		\end{aligned}
	\end{equation}
	We set
	\begin{align*}
		e_n(x):=\left(|\nabla u_n(x)|^{p-2}\nabla u_n(x)-|\nabla u(x)|^{p-2}\nabla u(x)\right)\cdot\nabla (u_n(x)-u(x)).
	\end{align*}
	Thanks to \eqref{LG5} we see that $e_n(x)\geq 0$ a.\,e.\,in $\Omega$. We split $\Omega$ by
	\begin{align*}
		S_n^k=\left\{x\in\Omega:\,\,|u_n(x)-u(x)|\leq k\right\}
		\quad\text{and}\quad
		G_n^k=\left\{x\in\Omega:\,\,|u_n(x)-u(x)|> k\right\},
	\end{align*}
	where $n,k \in \N$ are fixed. Taking $\theta\in(0,1)$ and using H\"older's inequality as well as the boundedness of $\{e_n\}_{n\in\N}$ in $\Lp{1}$ along with \eqref{LG6} gives
	\begin{align*}
		\int_\Omega e_n^\theta\,\mathrm{d}x
		&\leq\left(\int_{S_n^k} e_n\,\mathrm{d}x\right)^\theta|S_n^k|^{1-\theta}+\left(\int_{G_n^k} e_n\,\mathrm{d}x\right)^\theta|G_n^k|^{1-\theta}\\
		&\leq(kC)^\theta|S_n^k|^{1-\theta}+\widetilde{C}^\theta|G_n^k|^{1-\theta}.
	\end{align*}
	From this, noticing that $|G_n^k|\to0$ as $n\to\infty$, we get
	\begin{align*}
		0\leq\limsup_{n\to\infty}\int_\Omega e_n^\theta \,\mathrm{d}x\leq(kC)^\theta|\Omega|^{1-\theta}.
	\end{align*}
	Letting $k\to0^+$, we obtain that $e_n^\theta\to0$ in $L^1(\Omega)$ as $n\to\infty$. Thus, we may assume that $e_n(x)\to 0$ a.\,e.\,in $\Omega$. Applying \eqref{LG5} proves the assertion of the lemma.
\end{proof}

\begin{lemma}\label{Lps} 
	Let \eqref{assumptions} be satisfied and let $c<0$. Then, there exists $\lambda_0 > 0$ such that for any $\lambda\in(0,\lambda_0)$, the functional $J_\lambda$ satisfies the
	\textnormal{(PS)}$_c$ condition.
\end{lemma}

\begin{proof}
	Let $\lambda_0>0$ be sufficiently small such that
	\begin{equation}\label{main}
		|\Omega|\left(\frac{1}{q}-\frac{1}{p^*}\right)^{\frac{-p^*}{p^*-\vartheta}}
		\left[\lambda_0\left(\frac{1}{\vartheta}-\frac{1}{q}\right)\right]^{\frac{p^*}{p^*-\vartheta}}<S^{\frac{p^*}{p^*-p}},
	\end{equation}
	where $S$ is the best constant of the Sobolev embedding $W_0^{1,p}(\Omega)\hookrightarrow L^{p^*}(\Omega)$, namely
	\begin{equation}\label{S}
		S:=\inf_{u\in
			W_0^{1,p}(\Omega)\setminus\{0\}}\frac{\|\nabla u\|_p^p}{\|u\|_{p^*}^p}.
	\end{equation}

	Let $\lambda\in(0,\lambda_0)$ and let $\{u_n\}_{n\in\N}$ be a \textnormal{(PS)}$_c$ sequence in $\WH$. We first show that $\{u_n\}_{n\in\N}$ is bounded in $\WH$. Arguing by contradiction, going to a subsequence still denoted by $\{u_n\}_{n\in\N}$, we may suppose that $\lim\limits_{n\to\infty}\|u_n\|_{1,\mathcal{H},0}=\infty$ and $\|u_n\|_{1,\mathcal{H},0}\geq1$ for any $n\geq k$ with $k\in\mathbb N$ sufficiently large.
	Thus, according to Proposition \ref{proposition_embeddings}(ii), we get
	\begin{align*}
		J_\lambda(u_n)&-\frac{1}{p^*}\langle J'_\lambda(u_n), u_n\rangle\\
		&=\left(\frac{1}{p}-\frac{1}{p^*}\right)\|\nabla u_n\|_p^p
		+\left(\frac{1}{q}-\frac{1}{p^*}\right)\|\nabla u_n\|_{q,\mu}^q
		-\lambda\left(\frac{1}{\vartheta}-\frac{1}{p^*}\right)\|u_n\|_\vartheta^\vartheta\\
		& \geq\left(\frac{1}{q}-\frac{1}{p^*}\right)
		\varrho_{\mathcal H}(\nabla u_n)
		-\lambda C_\vartheta\|u_n\|_{1,\mathcal{H},0}^\vartheta.
	\end{align*}
	Thus, by \eqref{palais-smale} and Proposition \ref{proposition_modular_properties}(iv) there exist $c_1$, $c_2>0$ such that as $n\rightarrow\infty$,
	\begin{align*}
		c_1+c_2\|u_n\|_{1,\mathcal{H},0}+o(1)\geq \left(\frac{1}{q}-\frac{1}{p^*}\right)\|u_n\|_{1,\mathcal{H},0}^p
		-\lambda C_\vartheta\|u_n\|_{1,\mathcal{H},0}^\vartheta,
	\end{align*}
	which is a contradiction since $p^*>q>p>\vartheta>1$.
	
	Hence, $\{u_n\}_{n\in\N}$ is bounded in $\WH$. By Proposition \ref{proposition_embeddings}(ii), Lemma \ref{Lgradients}, Br\'{e}zis \cite[Theorem 4.9]{Brezis-2011} and the reflexivity of $W^{1,\mathcal H}_0(\Omega)$, there exists a subsequence, still denoted by $\{u_n\}_{n\in\N}$, and $u\in W^{1,\mathcal H}_0(\Omega)$ such that
	\begin{equation}\label{1}
		\begin{aligned}
			u_n& \weak u\quad \text{in }\WH,\quad
			& \nabla u_n &\weak \nabla u\quad \text{in }\left[L^{\mathcal H}(\Omega)\right]^N,\\
			\nabla u_n(x)& \to \nabla u(x) \quad \text{a.\,e.\,in }\Omega,
			\quad & u_n& \to u\quad \text{in }L^r(\Omega),\\
			u_n(x)&\to u(x)\quad \text{a.\,e.\,in }\Omega,
			\quad &\|u_n-u\|_{p^*}&\to\ell,
		\end{aligned}
	\end{equation}
	as $n\to\infty$  with $r\in[1,p^*)$. Let $A$ be the nodal set of the weight function $\mu(\cdot)$ given by
	\begin{align*}
		A=\left\{x\in\Omega\,: \,\mu(x)=0\right\}.
	\end{align*}
	As $\mu(\cdot)$ is a Lipschitz continuous function by \eqref{assumptions}, we know that $\Omega\setminus A$ is an open subset of $\R^N$. 
	
	Since the sequence $\{|\nabla u_n|^{p-2}\nabla u_n\}_{n\in\N}$ is bounded in $L^{p'}(\Omega)$, by \eqref{1} we get
	\begin{equation}\label{LG10}
		\lim_{n\to\infty}\int_\Omega|\nabla u_n|^{p-2}\nabla u_n\cdot\nabla u\,\,\mathrm{d}x=\|\nabla u\|_p^p.
	\end{equation}
	Because of the boundedness of $\{|\nabla u_n|^{q-2}\nabla u_n\}_{n\in\N}$ in $L^{q'}(\Omega\setminus A,\mu(x)\,\mathrm{d}x)$, by using \eqref{1} and Autuori-Pucci \cite[Proposition A.8]{Autuori-Pucci-2013} we conclude that
	\begin{align}\label{LG11}
		\begin{split}
			&\lim_{n\to\infty}\int_\Omega \mu(x)|\nabla u_n|^{q-2}\nabla u_n\cdot\nabla u\,\,\mathrm{d}x\\
			&= \lim_{n\to\infty}\int_{\Omega\setminus A} \mu(x)|\nabla u_n|^{q-2}\nabla u_n\cdot\nabla u\,\,\mathrm{d}x=\|\nabla u\|_{q,\mu}^q.
		\end{split}
	\end{align}
	Furthermore, using \eqref{3.10} and the Lemma of Br\'ezis-Lieb in Papageorgiou-Winkert \cite{Papageorgiou-Winkert-2018}, we obtain
	\begin{align}\label{Br}
		\begin{split}
			\|\nabla u_j\|_p^p-\|\nabla u_j-\nabla u\|_p^p&=\|\nabla u\|_p^p+o(1),\\
			\|\nabla u_j\|_{q,\mu}^q-\|\nabla u_j-\nabla u\|_{q,\mu}^q&=\|\nabla u\|_{q,\mu}^q+o(1),\\
			\|u_j\|_{p^*}^{p^*}-\|u_j-u\|_{p^*}^{p^*}&=\|u\|_{p^*}^{p^*}+o(1),
		\end{split}
	\end{align}
	as $n\to\infty$.  
	Thus, by \eqref{1}, \eqref{LG10} and \eqref{LG11}, we get
	\begin{align*}
			o(1)&=\langle J'_\lambda(u_n),u_n-u\rangle\\
			&=
			\int_\Omega\left(|\nabla u_n|^{p-2}\nabla u_n+\mu(x)|\nabla u_n|^{q-2}\nabla u_n\right)\cdot(\nabla u_n-\nabla u)\,\mathrm{d}x\\
			&\quad -\lambda\int_\Omega |u_n|^{r-2}(u_n-u)\,\mathrm{d}x-\int_\Omega |u_n|^{p^*-2}(u_n-u)\,\mathrm{d}x\\
			&=\|\nabla u_n\|_p^p-\|\nabla u\|_p^p
			+\|\nabla u_n\|_{q,\mu}^p-\|\nabla u\|_{q,\mu}^p
			-\|u_n\|_{p^*}^{p^*}-\|u\|_{p^*}^{p^*}+o(1)
		\end{align*}
	as $n\to\infty$. 
	Hence, by \eqref{1} and \eqref{Br} it follows that
	\begin{equation}\label{formula}
		\|\nabla u_n-\nabla u\|_p^p+\|\nabla u_n-\nabla u\|_{q,\mu}^q
		=\|u_n-u\|_{p^*}^{p^*}+o(1)=\ell^{p^*}+o(1)
	\end{equation}
	as $n\to\infty$.
	
	Now, assume for contradiction that $\ell>0$. By Proposition \ref{proposition_embeddings}(i), \eqref{S} and \eqref{formula}, we see that  $\ell^{p^*}\ge S\ell^p$ which implies
	\begin{equation}\label{ll}
		\ell\geq S^{\frac{1}{p^*-p}}.
	\end{equation}
	For any $n\in\mathbb N$ we have
	\begin{align*}
			&J_\lambda(u_n)-\frac{1}{q}\left\langle J'_\lambda(u_n),u_n\right\rangle\\
			&=\left(\frac{1}{p}-\frac{1}{q}\right)\|\nabla u_n\|_p^p
			-\lambda\left(\frac{1}{\vartheta}-\frac{1}{q}\right)\|u_n\|_\vartheta^\vartheta+\left(\frac{1}{q}-\frac{1}{p^*}\right)\|u_n\|^{p^*}_{p^*}.
	\end{align*}
	From this, as $n\to\infty$, by \eqref{palais-smale}, \eqref{1}, \eqref{Br}, H\"older's and Young's inequality, we obtain
	\begin{align*}
			c&\geq\left(\frac{1}{q}-\frac{1}{p^*}\right)\left(\ell^{p^*}+\|u\|^{p^*}_{p^*}\right)
			-\lambda\left(\frac{1}{\vartheta}-\frac{1}{q}\right)\|u\|^\vartheta_\vartheta\\
			&\geq\left(\frac{1}{q}-\frac{1}{p^*}\right)\left(\ell^{p^*}+\|u\|^{p^*}_{p^*}\right)
			-\lambda\left(\frac{1}{\vartheta}-\frac{1}{q}\right)|\Omega|^{\frac{p^*-\vartheta}{p^*}}\|u\|^\vartheta_{p^*}\\
			&\geq\left(\frac{1}{q}-\frac{1}{p^*}\right)\left(\ell^{p^*}+\|u\|^{p^*}_{p^*}\right)
			-\left(\frac{1}{q}-\frac{1}{p^*}\right)\|u\|^{p^*}_{p^*}\\
			&\quad-|\Omega|\left(\frac{1}{q}-\frac{1}{p^*}\right)^{-\frac{\vartheta}{p^*-\vartheta}}\left[\lambda\left(\frac{1}{\vartheta}-\frac{1}{q}\right)\right]^{\frac{p^*}{p^*-\vartheta}}.
	\end{align*}
	Finally, by \eqref{ll} we get
	\begin{align*}
		0>c\geq\left(\frac{1}{q}-\frac{1}{p^*}\right)S^{\frac{p^*}{p^*-p}} -
		|\Omega|\left(\frac{1}{q}-\frac{1}{p^*}\right)^{-\frac{\vartheta}{p^*-\vartheta}}
		\left[\lambda\left(\frac{1}{\vartheta}-\frac{1}{q}\right)\right]^{\frac{p^*}{p^*-\vartheta}}>0,
	\end{align*}
	where the last inequality follows from \eqref{main}. This gives a contradiction and shows that $\ell=0$. Hence, by \eqref{formula} and Proposition \ref{proposition_modular_properties}(v),  the result follows.
\end{proof}

Note that the energy functional $J_\lambda\colon \WH\to \R$ is not bounded from below. 
For this, we will use the treatment of Garc\'{\i}a Azorero-Peral Alonso \cite{Garcia-Azorero-Peral-Alonso-1991} following ideas from Figueiredo-Santos J\'{u}nior-Su\'{a}rez \cite{Figueiredo-Santos-Junior-Suarez-2015} and Zhang-Fiscella-Liang \cite{Zhang-Fiscella-Liang-2019}  in order to obtain a lower bound for a special truncated functional related to $J_\lambda$.
Let us define $\beta_\lambda \colon [0,\infty)\to\R$ by
\begin{align*}
	\beta_\lambda(t):=\frac{1}{q}t^q-\frac{\lambda}{\vartheta}  C_\vartheta t^\vartheta-\frac{1}{p^*}C_{p^*}t^{p^*},
\end{align*}
where $C_\vartheta$ and $C_{p^*}$ are the constants mentioned in Remark \ref{remark-constants}.
Since $\vartheta<q$ we see that $\beta_\lambda(t)<0$ for $t$ near zero and due to $1<\vartheta<p<q<p^*$ there exists $\lambda_1>0$ such that $\beta_\lambda$ attains its positive maximum for any $\lambda \in(0,\lambda_1)$. Let $R_0(\lambda)$ and  $R_1(\lambda)$ be the unique roots of $\beta_\lambda$ such that $0<R_0(\lambda)<R_1(\lambda)$.

%
Then, we have the following claim.

\noindent {\bf Claim:} {\em $R_0(\lambda)\to 0$ as $\lambda\to 0$}.

From $\beta_\lambda(R_0(\lambda))=0$ and $\beta_\lambda'(R_0(\lambda))>0$ we have
\begin{align}\label{root-1}
	\frac{1}{q}R_0(\lambda)^q=\frac{\lambda}{\vartheta}  C_\vartheta R_0(\lambda)^\vartheta+\frac{1}{p^*}C_{p^*}R_0(\lambda)^{p^*}
\end{align}
and
\begin{align}\label{root-2}
	R_0(\lambda)^{q-1}>\lambda C_\vartheta R_0(\lambda)^{\vartheta-1}+C_{p^*}R_0(\lambda)^{p^*-1}
\end{align}
for any $\lambda \in (0,\lambda_1)$. From \eqref{root-1} we know that $R_0(\lambda)$ is bounded since $\vartheta<q<p^*$. Suppose that $R_0(\lambda)\to R>0$ as $\lambda \to 0$. Then we obtain from \eqref{root-1} and \eqref{root-2}
\begin{align*}
	\frac{1}{q}R^q= \frac{1}{p^*}C_{p^*} R^{p^*}
	\quad\text{and}\quad 
	R^{q-1}\geq C_{p^*}R^{p^*-1},
\end{align*}
which is a contradiction since $q<p^*$. This proves the Claim.

From the Claim, there exists $\lambda_2>0$ such that $R_0(\lambda)<1$ for any $\lambda\in(0,\lambda_2)$. This implies that $R_0(\lambda)<\min\{R_1(\lambda),1\}$. Thus, for any $\lambda\in(0,\min\{\lambda_1,\lambda_2\})$ we choose a $C^\infty$-function $\tau\colon [0,\infty)\to [0,1]$ such that
\begin{align}\label{def-tau}
	\tau(t):=
	\begin{cases}
		1 & \text{if } t \in [0,R_0(\lambda)],\\
		0 & \text{if }t \in [\min\{R_1(\lambda),1\},\infty).
	\end{cases}
\end{align}
Then, we can introduce following truncated energy functional
\begin{align*}
	\widehat{J}_\lambda(u):=\frac{1}{p}\|\nabla u\|_p^p+\frac{1}{q}\|\nabla u\|_{q,\mu}^q-\frac{\lambda}{\vartheta}\|u\|_{\vartheta}^{\vartheta}-\frac{1}{p^{*}}\|u\|_{p^*}^{p^*}\tau(\|u\|_{1,\mathcal{H},0}).
\end{align*}
It is clear that $\widehat{J}_\lambda \in C^1(\WH,\R)$ is coercive and bounded from below. Also, note that if $\|u\|_{1,\mathcal{H},0}\leq R_0(\lambda)< \min\{R_1(\lambda),1\}$, then $\widehat{J}_\lambda(u)=J_\lambda(u)$.

Thus, we have the following technical result.

\begin{lemma}\label{lemma_truncated_functional}
	Let \eqref{assumptions} be satisfied. Then there exists $\lambda^*>0$ such that for any $\lambda \in (0,\lambda^*)$ the following hold:
	\begin{enumerate}
		\item[\textnormal{(i)}]
		If $\widehat{J}_\lambda(u) < 0$, then $\|u\|_{1,\mathcal{H},0} <R_0(\lambda)$ and $J_\lambda(v)=\widehat{J}_\lambda(v)$ for any $v$ in a sufficiently small neighborhood of $u$;
		\item[\textnormal{(ii)}]
		$\widehat{J}_\lambda$ fulfills a local  \textnormal{(PS)}$_c$ condition for $c<0$.
	\end{enumerate}
\end{lemma}

\begin{proof}
	Let $\lambda^* \leq \min\{\lambda_0,\lambda_1,\lambda_2,\lambda_3\}$, where $\lambda_0$ is given in Lemma \ref{Lps}, $\lambda_1$ and $\lambda_2$ are chosen for the definition of $\tau$ in \eqref{def-tau}, while $\lambda_3:=\frac{\vartheta}{C_\vartheta q}$ with $C_\vartheta$ mentioned in Remark \ref{remark-constants}. Let $\lambda \in (0,\lambda^*)$.
	
\noindent	\textnormal{(i)} Let $\widehat{J}_\lambda(u) < 0$. We distinguish two different cases.

	{\bf Case 1: } $\|u\|_{1,\mathcal{H},0} \geq 1$.

	This case cannot occur. Indeed, by the definition of $\tau$ in \eqref{def-tau} we know that $\tau(\|u\|_{1,\mathcal{H},0})=0$. Therefore, by Propositions \ref{proposition_modular_properties}(iv)  and \ref{proposition_embeddings}(ii) we get that
	\begin{align}\label{estimate-below10}
		\begin{split}
			\widehat{J}_\lambda(u)
			&=\frac{1}{p}\|\nabla u\|_p^p+\frac{1}{q}\|\nabla u\|_{q,\mu}^q-\frac{\lambda}{\vartheta}\|u\|_{\vartheta}^{\vartheta}
			-\frac{1}{p^{*}}\|u\|_{p^*}^{p^*}\tau(\|u\|_{1,\mathcal{H},0})\\
			&\geq \frac{1}{q}\|u\|_{1,\mathcal{H},0}^p -\frac{\lambda}{\vartheta} C_\vartheta \|u\|_{1,\mathcal{H},0}^\vartheta\\
			&=\phi_\lambda (\|u\|_{1,\mathcal{H},0}),
		\end{split}
	\end{align}
	where $\phi_\lambda \colon [1,\infty)\to\R$ is given by
	\begin{align*}
		\phi_\lambda(t):=\frac{1}{q}t^p-\frac{\lambda}{\vartheta}  C_\vartheta t^\vartheta.
	\end{align*}
	It is clear that $\phi_\lambda$ has a global minimum point at 
	\begin{align*}
		t_0=\l(\lambda \frac{C_\vartheta q}{p}\r)^{\frac{1}{p-\vartheta}}
	\end{align*}
	with
	\begin{align*}
		\phi_\lambda(t_0)=\frac{1}{q}\l(\lambda \frac{C_\vartheta q}{p}\r)^{\frac{p}{p-\vartheta}} \l(1-\frac{p}{\vartheta}\r)<0,
	\end{align*}
	since $\vartheta<p$. 
	
				
	
	We point out that $\phi_\lambda(t)\geq 0$ if and only if $t \geq \l(\lambda \frac{C_\vartheta q}{\vartheta}\r)^{\frac{1}{p-\vartheta}}$. Hence, choosing $\lambda \leq \lambda_3=\frac{\vartheta}{C_\vartheta q}$ we have $\displaystyle\min_{t\in[1,\infty]}\phi_\lambda(t)\geq 0$ which yields, joint with \eqref{estimate-below10}, that $\widehat{J}_{\lambda}(u) \geq 0$ for any $\|u\|_{1,\mathcal{H},0} \geq 1$. This gives the desired contradiction.

	{\bf Case 2: } $\|u\|_{1,\mathcal{H},0} < 1$.

	By Propositions \ref{proposition_modular_properties}(iii) and \ref{proposition_embeddings}(ii) we get
	\begin{align*}
		\begin{split}
			\widehat{J}_\lambda(u)
			&\geq \frac{1}{q}\|u\|_{1,\mathcal{H},0}^q -\frac{\lambda}{\vartheta} C_\vartheta \|u\|_{1,\mathcal{H},0}^\vartheta
			-\frac{1}{p^*}C_{p*}\|u\|_{1,\mathcal{H},0}^{p^*}\tau(\|u\|_{1,\mathcal{H},0})\\
			&=\widehat{\beta}_\lambda (\|u\|_{1,\mathcal{H},0}),
		\end{split}
	\end{align*}
	where
	\begin{align*}
		\widehat{\beta}_\lambda(t):=\frac{1}{q}t^q-\frac{\lambda}{\vartheta}  C_\vartheta t^\vartheta-\frac{1}{p^*}C_{p^*}t^{p^*}\tau(t).
	\end{align*}
	Since $0\leq\tau\leq1$, we note that 
	\begin{align}\label{beta-1-hat}
	\widehat{\beta}_\lambda(t)\geq\beta_\lambda(t)\geq0\quad\mbox{for any }t \in[R_0(\lambda),\min\{R_1(\lambda),1\}],
	\end{align}
	where the last inequality follows by the construction of the roots $R_0(\lambda)$ and $R_1(\lambda)$ for $\beta_\lambda$. 
	
	Hence, if $\min\{R_1(\lambda),1\}=1$, then from $\widehat{J}_\lambda(u) < 0$ and \eqref{beta-1-hat} we obtain that $\|u\|_{1,\mathcal{H},0}<R_0(\lambda)$. 
	
	While, if $\min\{R_1(\lambda),1\}=R_1(\lambda)$, considering  $R_1(\lambda)<\|u\|_{1,\mathcal{H},0}<1$ and arguing similarly to \eqref{estimate-below10} we get
	\begin{align*}
	\widehat{J}_\lambda(u)\geq\widehat{\phi}_\lambda (\|u\|_{1,\mathcal{H},0})\qquad\mbox{ with }\ \ \widehat{\phi}_\lambda(t):=\frac{1}{q}t^q-\frac{\lambda}{\vartheta}  C_\vartheta t^\vartheta,
	\end{align*}
	from which we can proceed exactly as in Case 1 to get a contradiction. Considering  $R_0(\lambda)<\|u\|_{1,\mathcal{H},0}\leq R_1(\lambda)$, from $\widehat{J}_\lambda(u) < 0$ and \eqref{beta-1-hat} we get another contradiction. Hence, we obtain again $\|u\|_{1,\mathcal{H},0}<R_0(\lambda)$, completing the first part of (i). 
	
	Moreover, we observe that $\widehat{J}_\lambda(v)=J_\lambda(v)$ for any $\|v-u\|_{1,\mathcal{H},0}<R_0(\lambda)-\|u\|_{1,\mathcal{H},0}$, concluding the proof of (i).

\noindent	\textnormal{(ii)} Note that any Palais-Smale sequence for $\widehat{J}_\lambda$ is bounded since $\widehat{J}_\lambda$ is coercive. Applying Lemma \ref{Lps}  shows that we have a local Palais-Smale condition for $J_\lambda \equiv \widehat{J}_\lambda$ at any level $c<0$.
\end{proof}

Now, we are going to construct an appropriate mini-max sequence of negative critical values for the functional $\widehat{J}_\lambda$. 

\begin{lemma}\label{lemma_mini-max-critical-values}
	Let \eqref{assumptions} be satisfied and let $\lambda\in(0,\lambda_2)$, where $\lambda_2$ is chosen for the definition of $\tau$ in \eqref{def-tau}. For any $n\in\N$ there exists $\eps=\eps(\lambda,n)>0$ such that
	\begin{align*}
		\gamma \l(\widehat{J}_\lambda^{-\eps}\r) \geq n.
	\end{align*}
	where $\widehat{J}_\lambda^{-\eps}=\l\{u \in \WH\,:\, \widehat{J}_\lambda(u)\leq -\eps\r\}$.
\end{lemma} 

\begin{proof}
	Let $\lambda\in(0,\lambda_2)$ and $n\in \N$ be fixed and let $Y_n$ be an $n$-dimensional subspace of $\WH$. By Proposition \ref{proposition_embeddings}(ii) we have that $Y_n \hookrightarrow \Lp{\vartheta}$ is continuously embedded. Hence, the norms of $\WH$ and $\Lp{\vartheta}$ are equivalent on $Y_n$. In particular, there exists a positive constant $C(n)$ which depends only on $n$ such that
	\begin{align}\label{equivalent-norm}
		-C(n)\|u\|_{1,\mathcal{H},0}^{\vartheta}\geq -\|u\|_{\vartheta}^{\vartheta}\quad\text{for any }u \in Y_n.
	\end{align}
	Therefore, using \eqref{equivalent-norm} and Proposition \ref{proposition_modular_properties}(iii), for any $u\in Y_n$ with $\|u\|_{1,\mathcal{H},0}\leq R_0(\lambda)<1$ we have 
	\begin{align}\label{estimate-finitel-dimensional}
		\widehat{J}_\lambda(u)
		\leq \frac{1}{p}\|u\|_{1,\mathcal{H},0}^p-\frac{\lambda}{\vartheta}C(n)\|u\|_{1,\mathcal{H},0}^\vartheta.
	\end{align}
	Now, let $r$ and $R$ be two positive constants such that
	\begin{align}\label{choice-parameter}
		r<R<\min\l\{R_0(\lambda),\l(\frac{\lambda C(n)p}{\vartheta}\r)^{\frac{1}{p-\vartheta}}\r\}
	\end{align}
	and let
	\begin{align*}
		\mathbb{S}_n=\l\{u\in Y_n \,:\,\|u\|_{1,\mathcal{H},0}=r\r\}.
	\end{align*}
	It is clear that $\mathbb{S}_n$ is homeomorphic to the $(n-1)$-dimensional sphere $S^{n-1}$. Thus, from Proposition \ref{proposition-genus}(iv) we know that $\gamma(\mathbb{S}_n)=n$. Furthermore, from \eqref{estimate-finitel-dimensional} and \eqref{choice-parameter} we obtain
	\begin{align*}
		\widehat{J}_\lambda(u)
		\leq r^\vartheta\l(\frac{1}{p}r^{p-\vartheta}-\frac{\lambda}{\vartheta}C(n)\r)\leq R^\vartheta\l(\frac{1}{p}R^{p-\vartheta}-\frac{\lambda}{\vartheta}C(n)\r)<0.
	\end{align*}
	Hence, we find a constant $\eps>0$ such that $\widehat{J}_\lambda(u)<-\eps$ for any $u\in\mathbb{S}_n$, that is, $\mathbb{S}_n\subset \widehat{J}_\lambda^{-\eps}$ and so, by Proposition \ref{proposition-genus}(i),
	\begin{align*}
		\gamma \l(\widehat{J}_\lambda^{-\eps}\r) \geq \gamma (\mathbb{S}_n)=n.
	\end{align*}
\end{proof}

Next we define for any $n\in\N$ the sets
\begin{align*}
	\Sigma_n&=\l\{A\subset \WH\setminus\{0\}\,:\,A \text{ is closed, }A=-A \text{ and }\gamma(A)\geq n\r\},\\
	K_c&=\l\{u \in \WH\,:\,\widehat{J}\,'_\lambda(u)=0 \text{ and } \widehat{J}_\lambda(u)=c\r\}
\end{align*}
and the number
\begin{align*}
	c_n=\inf_{A\in \Sigma_n} \sup_{u\in A} \widehat{J}_\lambda(u).
\end{align*}
Clearly, $c_n \leq c_{n+1}$ for any $n\in\mathbb N$.

\begin{lemma}\label{lemma_cn-negative}
	Let \eqref{assumptions} be satisfied and let $\lambda\in(0,\lambda_2)$, where $\lambda_2$ is chosen for the definition of $\tau$ in \eqref{def-tau}. For any $n\in\N$, the number $c_n$ is negative.
\end{lemma} 

\begin{proof}
	Let $\lambda\in(0,\lambda_2)$ and $n\in \N$ be fixed. From Lemma \ref{lemma_mini-max-critical-values} we know there exists $\eps>0$ such that $\gamma \l(\widehat{J}_\lambda^{-\eps}\r) \geq n$. Also, since $\widehat{J}_\lambda$ is even and continuous, we know that $\widehat{J}_\lambda^{-\eps} \in\Sigma_n$. From $\widehat{J}_\lambda(0)=0$ we have $0\not\in \widehat{J}_\lambda^{-\eps}$. Since $\sup_{u\in \widehat{J}_\lambda^{-\eps}} \widehat{J}_\lambda(u)\leq -\eps$ and $\widehat{J}_\lambda$ is bounded from below, we obtain
	\begin{align*}
		-\infty < c_n= \inf_{A\in \Sigma_n} \sup_{u\in A} \widehat{J}_\lambda(u)\leq \sup_{u\in \widehat{J}_\lambda^{-\eps}} \widehat{J}_\lambda(u)\leq -\eps <0.
	\end{align*}
\end{proof}

Next we have the following lemma. Note that in Garc\'{\i}a Azorero-Peral Alonso \cite{Garcia-Azorero-Peral-Alonso-1991} a similar lemma was proved. In fact, the proof of Lemma \ref{lemma-critical-values} works similarly. For sake of clarity we give here the proof.
\begin{lemma}\label{lemma-critical-values}
	Let \eqref{assumptions} be satisfied, let $\lambda \in (0,\lambda^*)$, where $\lambda^*>0$ is as in Lemma \ref{lemma_truncated_functional}, and let $n\in\mathbb N$. If $c=c_n=c_{n+1}=\ldots=c_{n+l}$ for some $l\in\mathbb N$, then
	\begin{align*}
		\gamma(K_c)\geq l+1.
	\end{align*}
\end{lemma}

\begin{proof}
	Let $\lambda\in(0,\lambda^*)$. Since from Lemma \ref{lemma_cn-negative} we have that $c=c_n=c_{n+1}=\ldots=c_{n+l}$ is negative, then by Lemma \ref{lemma_truncated_functional}(ii) it easily follows that $K_c$ is compact.

	Let us assume by contradiction that $\gamma(K_c)\leq l$. Then, by Proposition \ref{proposition-genus}(v) there exists $\delta>0$ such that $\gamma(N_\delta(K_c))=\gamma(K_c)\leq l$, where 
	
	\begin{align*}
		N_\delta(K_c)=\left\{u\in \WH:\,\, \dist(u,K_c)\leq\delta\right\}.
	\end{align*}
	By the deformation theorem, see for example Benci \cite[Theorem 3.4]{Benci-1982}, there exist $\eps \in (0,-c)$ and an odd homeomorphism $\eta\colon \WH\to \WH$ such that
	\begin{align}\label{critical-1}
		\eta\l(\widehat{J}_\lambda^{c+\eps}\setminus N_\delta(K_c)\r)\subset \widehat{J}_\lambda^{c-\eps}.
	\end{align}
	While, by the definition of $c=c_{n+l}$ there exists $A\in\Sigma_{n+l}$ such that $\sup_{u\in A}\widehat{J}_\lambda(u)<c+\eps$, that is $A\subset\widehat{J}_\lambda^{c+\eps}$, and so by \eqref{critical-1}
	\begin{align}\label{critical-2}
		\eta\l(A\setminus N_\delta(K_c)\r)\subset
		\eta\l(\widehat{J}_\lambda^{c+\eps}\setminus N_\delta(K_c)\r)\subset \widehat{J}_\lambda^{c-\eps}.
	\end{align}
	On the other hand, from Proposition \ref{proposition-genus}\textnormal{(i)} and \textnormal{(iii)} we have
	\begin{align*}
		\gamma(\eta(\overline{A \setminus N_\delta(K_{c})}))\geq\gamma(\overline{A \setminus N_\delta(K_{c})}) \geq \gamma(A)-\gamma(N_\delta(K_{c}))\geq n.
	\end{align*}
	Hence, we conclude that $\eta(\overline{A \setminus N_\delta(K_{c})}) \in \Sigma_n$ and so
	\begin{align*}
		\sup_{u\in \eta(\overline{A \setminus N_\delta(K_{c})})} \widehat{J}_\lambda(u) \geq c_n=c,
	\end{align*}
	which contradicts \eqref{critical-2}.
\end{proof}

\begin{proof}[Proof of Theorem \ref{main_theorem}]
	
Let $\lambda\in(0,\lambda^*)$, where $\lambda^*>0$ is as in Lemma \ref{lemma_truncated_functional}. 
By Lemma \ref{lemma_cn-negative} we have $c_n<0$. Hence, from Lemma \ref{lemma_truncated_functional}(ii) we know that the functional $\widehat{J}_\lambda$ satisfies the Palais-Smale condition at level $c_n<0$. Thus, $c_n$ is a critical value of $\widehat{J}_\lambda$ for any $n\in\N$, see for example Rabinowitz \cite{Rabinowitz-1986}.

We consider two situations.
If $-\infty<c_1<c_2<\ldots<c_n<c_{n+1}<\ldots$, then $\widehat{J}_\lambda$ admits infinitely many critical values. 
If there exist $n$, $l\in\mathbb N$ such that $c_n=c_{n+1}=\ldots=c_{n+l}=c$, then $\gamma(K_c)\geq l+1\geq2$ by Lemma \ref{lemma-critical-values}. Thus, the set $K_c$ has infinitely many points, see Rabinowitz \cite[Remark 7.3]{Rabinowitz-1986}, which are infinitely many critical values for $\widehat{J}_\lambda$ by Lemma \ref{lemma_truncated_functional}(ii).

Then, by Lemma \ref{lemma_truncated_functional}(i) we get infinitely many negative critical values for $J_\lambda=\widehat{J}_\lambda$ and so problem \eqref{problem} has infinitely many weak solutions.
\end{proof}

\section*{Acknowledgments}
C.\,Farkas was supported by the National Research, Development and Innovation Fund of Hungary, financed under the K\_18 funding scheme, Project No.\,127926.

A.\,Fiscella is member of the {Gruppo Nazionale per l'Analisi Ma\-tema\-tica, la Probabilit\`a e
	le loro Applicazioni} (GNAMPA) of the {Istituto Nazionale di Alta Matematica ``G. Severi"} (INdAM).
A.\,Fiscella realized the manuscript within the auspices of the INdAM-GNAMPA project titled "Equazioni alle derivate parziali: problemi e modelli" (Prot\_20191219-143223-545), of the FAPESP Project titled "Operators with non standard growth" (2019/23917-3) and of the FAPESP Thematic Project titled "Systems and partial differential equations" (2019/02512-5).

\end{document}